\documentclass{article}
\usepackage{amsmath,amssymb, amsfonts, amsbsy, latexsym, color,bbm,amsthm}
\usepackage{graphicx}
\usepackage{fullpage}
\usepackage{hyperref}

\theoremstyle{plain}
\newtheorem{thm}{Theorem}[section]
\newtheorem{cor}[thm]{Corollary}
\newtheorem{lem}[thm]{Lemma}
\newtheorem{prop}[thm]{Proposition}

\theoremstyle{definition}
\newtheorem{defn}[thm]{Definition}

\theoremstyle{remark}

\newtheorem{ex}[thm]{Example}
\newtheorem{question}[thm]{Question}
\newtheorem{conj}[thm]{Conjecture}

\newcommand{\cH}{\mathcal{H}}

\newcommand{\cB}{\mathcal{B}}

\newcommand{\bZ}{\mathbb{Z}}
\newcommand{\bN}{\mathbb{N}}
\newcommand{\bR}{\mathbb{R}}
\newcommand{\bC}{\mathbb{C}}

\newcommand{\bF}{\mathbb{F}}
\newcommand{\bT}{\mathbb{T}}

\newcommand{\tr}{\operatorname{tr}}
\newcommand{\spark}{\operatorname{spark}}
\newcommand{\rar}{\rightarrow}

\newcommand{\set}[2]{\left\{\,  #1  \,\, \middle\vert \,\, #2 \, \right\}  }
\newcommand{\QR}{\operatorname{QR}}
\newcommand{\NR}{\operatorname{NR}}

\newcommand{\U}{\operatorname{U}}

\newcommand{\ip}[2]{\left\langle#1,#2\right\rangle}
\newcommand{\norm}[1]{\left\Vert#1\right\Vert}
\newcommand{\abs}[1]{\left\vert#1\right\vert}
\newcommand{\absip}[2]{\left\vert\langle#1,#2\rangle\right\vert}

\newcommand{\beq}{\begin{equation}}
\newcommand{\eeq}{\end{equation}}

\DeclareMathOperator*{\TP}{TP}

\DeclareMathOperator*{\Sym}{Sym}

\begin{document}

\title{$k$-Homogeneous Equiangular Tight Frames}

\author{Emily J.\ King\thanks{\texttt{emily.king@colostate.edu}, Department of Mathematics, Colorado State University, Fort Collins, CO}}

\maketitle

\begin{center}
\textit{This is dedicated to David Larson, a great mentor and researcher.}
 \end{center} 

 \begin{abstract}
We consider geometric and combinatorial characterizations of equiangular tight frames (ETFs), with the former concerning homogeneity of the vector and line symmetry groups and the latter the matroid structure. We introduce the concept of the bender of a frame, which is the collection of short circuits, which in turn are the dependent subsets of frame vectors of minimum size. We also show that ETFs with $k$-homogeneous line symmetry groups have benders which are $k$-designs.  Paley ETFs are a known class of ETFs constructed using number theory.  We determine the line and vector symmetry groups of all Paley ETFs and show that they are $2$-homogeneous. We additionally characterize all $k$-homogeneous ETFs for $k\geq 3$. Finally, we revisit David Larson's AMS Memoirs \emph{Frames, Bases, and Group Representations} \cite{HL00} coauthored with Deguang Han and  \emph{Wandering Vectors for Unitary Systems and Orthogonal Wavelets} \cite{DaL98} coauthored with Xingde Dai with a modern eye and focus on finite-dimensional Hilbert spaces. \textbf{Keywords}: equiangular tight frames, ETFs, $k$-homogeneous, $k$-transitive, Paley difference sets, $t$-designs, matroid circuits \textbf{MSC2020}: 42C15, 20B25, 05B10, 05B35
\end{abstract}

\section{Introduction and Preliminaries}

Frames are generalizations of orthonormal bases that have found wide applications in signal and image processing, compressed sensing, phase retrieval, and even characterizing emergent phenomenon in neural networks \cite{CaKBook,KoCh07,FoRa13,papyan2020prevalence}.  Frames were originally introduced in~\cite{DS52} as a tool to loosen the definition of Fourier series.  Frames over finite dimensional Hilbert spaces (see, e.g.,~\cite{Waldron18}), in particular so-called \emph{equiangular tight frames (ETFs)} are intimately related to open problems in quantum information theory (like \emph{Zauner's conjecture} \cite{Zauner1999,Zauner2011}) and combinatorial design theory (like the \emph{Gillespie conjecture} \cite{Gill18}\footnote{It is possible that this conjecture is folklore.  See, e.g., p.\ 263 of \cite{godsil2001algebraic}.}).  In this paper, we introduce the concept of the \emph{bender} of a frame, which is the collection of all linearly dependent subsets of frame vectors of minimum size (called \emph{short circuits}). 
Inspired by the complete classification of both ETFs with $2$-transitive symmetry groups \cite{IvMi18,dempwolff2023trans,iverson2024doubly} and  $k$-homogeneous groups which are not $k$-transitive~\cite{kantor1972khomo}, we prove results concerning the symmetry group and bender of an infinite class of ETFs (\emph{Paley ETFs}). We also characterize all ETFs which have $k$-homogeneous symmetry groups for $k \geq 3$.  Finally, we re-examine two of Dave Larson's fundamental manuscripts from the turn of the millennium \cite{HL00,DaL98} in the context of frames in finite dimensions.

The author first learned about frames in 2002 in Dave Larson's Research Experiences for Undergraduates (REU) on Frames and Wavelets, which ran for many years and inspired numerous young mathematicians.  She has also used his book \emph{Frames for undergraduates} \cite{han2007frames}---cowritten with Deguang Han, Keri Kornelson, and Eric Weber, which came out of the experience running the REU---when working with students from high school to master's level. Almost a quarter of a century later, the REU run by Dave Larson still has a large impact on her research program (and the research of many other alumni of the REU).

The basics of equiangular tight frames are covered in Section~\ref{sec:ETFs}, followed by sections on their symmetries Section~\ref{sec:geomETF} and linear dependencies Section~\ref{sec:combETF}. In Definition~\ref{defn:bnder} in Section~\ref{sec:combETF}, a new combinatorial concept is introduced: the \emph{bender}, which consists of the \emph{short circuits} of an ETF, which are the dependent subsets of minimum size.  Section~\ref{sec:Paley} contains results concerning Paley ETFs, with the main result being Theorem~\ref{thm:Paleydbhm}, in which the symmetry groups of all Paley ETFs are given.  Section~\ref{sec:homoETF} concerns $k$-homogeneous ETFs, including Theorem~\ref{thm:khomETF} which characterizes $k$-homogeneous ETFs for $k \geq 3$.  Finally, some key results in Dave Larson's Memoirs \cite{HL00,DaL98} are contextualized within the framework of modern research on finite frames in Section~\ref{sec:memoirs}.  This section also includes a list of some questions and conjectures.

\subsection{Equiangular Tight Frames}\label{sec:ETFs}
For any separable Hilbert space $\cH$ and any orthonormal basis $\Phi = (\varphi_j)_{j\in J}$, Parseval's equality yields that for all $x \in \cH$
\[
\norm{x}^2 = \sum_{j \in J} \absip{x}{\varphi_j}^2.
\]
What might be surprising is that there are collections of vectors $\Phi$ which are not orthonormal bases which also satisfy that equality exactly (\emph{Parseval frames}), in a scaled version (\emph{tight frames}), or in a loosened version (\emph{frames}).  In infinite dimensions, even being a frame requires significant structure, and there is an immense amount of research about such vectors (see, e.g., \cite{Chr03v2}). 
\begin{defn}\label{defn:frame}
For a separable Hilbert space $\cH$ and $\Phi = (\varphi_j)_{j\in J}$ in $\cH$, $\Phi$ is a \emph{tight frame} if there exists $A>0$ such that for all $x \in \cH$
\[
A\norm{x}^2 = \sum_{j \in J} \absip{x}{\varphi_j}^2.
\]
$A$ is called the \emph{frame bound}. If $\norm{\varphi_j} = \norm{\varphi_k}$ for all $j, k \in J$, we call $\Phi$ \emph{equal--norm}.
\end{defn}
In finite dimensions, any spanning set is a frame (in general, not tight), and one often asks for additional structure, like being an equiangular tight frame (defined below).  For finite sets $D, J$, we will use $\bF^{D \times J}$ to denote the vector space of functions from $D \times J$ to $\bF$, which may be represented as matrices once orders have been fixed for $D$ and $J$.  Also, $\bF^{d \times n}$ for $d, n \in \bN$ is shorthand for $\bF^{\{1, 2, \hdots, d\} \times \{1, 2, \hdots, n\}}$.  We similarly define the Hilbert spaces $\bF^D$ and $\bF^d$. 
In this section only, we will use the notation $\bF$ (without a subscript) to mean $\bR$ or $\bC$. 
\begin{defn}
Let $\Phi = (\varphi_j)_{j \in J} \in \bF^{D\times J}$ for $\abs{D}=d$, $\abs{J}=n$, and $d \leq n$. We call $\Phi$ \emph{equiangular} if there exist $a\leq b$ with $b> 0$ such that 
\[
\absip{\varphi_j}{\varphi_k} =(b-a)\delta_{j,k} + a.
\]
If additionally, $\Phi$ is a tight frame, then we call $\Phi$ an \emph{equiangular tight frame (ETF)} or a \emph{$(d,n)$-ETF}.

\end{defn}
Note an ETF is by definition equal--norm.
We will employ the common abuse of notation using $\Phi$ to both mean the frame itself and the short fat matrix with columns the frame vectors. With that notation, $\Phi$ is a tight frame precisely when $\Phi \Phi^\ast = A I$.
Equiangular tight frames are optimally robust to noise and up to two erasures \cite{GKK01,StH03,HoPa04}, are important in quantum information theory~\cite{Zauner1999,Zauner2011}, and provide optimal projective codes~\cite{JKM19}.  As mathematical objects which interact with a variety of fields, there are a number of conjectures about their existence or the lack thereof: Zauner's conjecture~\cite{Zauner1999,Zauner2011}, the Gillespie conjecture~\cite{Gill18}, the $d \times 2d$ conjecture~\cite{fallon2023optimal,iverson2024more}, and the Fickus-Jasper conjecture~\cite{fickus2019GDD}.

For many choices of $d$, $n$, $\bF$, there does not exist an ETF. (See, e.g., \cite{FM15}.)  When $n=d$, the ETFs are simply equal-norm orthogonal bases.  Up to equivalence (see below) when $n=d+1$, the vectors span lines which go through vertices of a regular simplex.
\begin{defn}
If $\Phi$ is a $(d,d+1)$-ETF with $d >1$, we call it a \emph{simplex ETF}.  If further, the inner product of any two distinct vectors in a simplex ETF is negative (necessarily the same negative number since they are equiangular), then we call it a \emph{regular simplex ETF}.
\end{defn}

In a certain sense, every simplex ETF is equivalent to a regular simplex ETF. Let $\U(D)$ denote linear operators $\bF^D \rar \bF^D$ which are orthogonal when $\bF=\bR$ and unitary when $\bF=\bC$. Then for any $(D,J)$-ETF $(\varphi_j)_{j \in J}$  with $\abs{J}=\abs{D}+1$, there is always a $U \in \U(D)$ and a choice of scalars $\eta_j \in \U(1)$ such that
\[
\left(U \varphi_j \eta_j\right)_{j \in J}
\]
have real entries and point to vertices of a regular $d$-simplex centered at the origin. In general, one may ask if two sequences of vectors are fundamentally ``the same'', i.e., span the same set of lines up to a global change of basis. Fortunately, there is a relatively simple test.
\begin{defn}
Assume that $\Phi = (\varphi_j)_{j\in J},\Psi = (\psi_j)_{j \in J}   \in \bF^{D\times J}$.  We say that the vectors are \emph{switching equivalent} (or \emph{projectively unitarily equivalent} or \emph{isometrically isomorphic}) if there exists a 
$U \in \U(D)$ and a choice of scalars $\eta_j \in \U(1)$ such that
\[
\psi_j = U \varphi_j \eta_j, \quad j \in J.
\]
\end{defn}
We note that some authors also allow permutation of the indices in their definition of switching equivalence. The test for switching equivalence has been rediscovered in different forms over many years (e.g., \cite{seidel1976survey,gallagher1977orthogonal,brehm1990shape,AFF11,ChWa16}).
\begin{prop}\label{prop:switch}
Let $\Phi = (\varphi_j)_{j\in J},\Psi = (\psi_j)_{j \in J}   \in \bF^{D\times J}$ be ETFs. They are switching equivalent if and only if  for $j \neq k \neq \ell \neq j$
\[
\ip{\varphi_j}{\varphi_k}\ip{\varphi_k}{\varphi_{\ell}}\ip{\varphi_{\ell}}{\varphi_j}=\ip{\psi_j}{\psi_k}\ip{\psi_k}{\psi_{\ell}}\ip{\psi_{\ell}}{\psi_j}.
\]
\end{prop}
We will use the notation $\TP(j,k,\ell)=\ip{\varphi_j}{\varphi_k}\ip{\varphi_k}{\varphi_{\ell}}\ip{\varphi_{\ell}}{\varphi_j}$, where $\Phi$ will be clear from context.

So, any ETF with $d=1$, $n-1$, or $n$ is in a sense trivial.  
An ETF with $d=1$ is just a sequence of unimodular constants, with $d=n-1$ is a simplex ETF, and with $d=n$ an orthonormal basis.
\begin{defn}
We call a $(d,n)$-ETF with $d=1$, $n-1$, or $n$ a \emph{trivial ETF}.  
Otherwise, it is \emph{nontrivial}.
\end{defn}

Every nontrivial $(d,n)$-ETF over $\bF$ has an (actually, infinitely many) associated $(n-d,n)$-ETF over $\bF$ called a \emph{Naimark complement}. Naimark complements of Parseval frames (not in general equiangular and allowing infinite dimensional Hilbert spaces) were first studied by Han and Larson~\cite{HL00}, where they were called \emph{complementary frames}.  It was later realized that these results followed from Naimark's dilation theorem concerning the lifting of positive operator valued measures~\cite{neumark1943representation}.
\begin{prop}\label{prop:Naimark}
Let $\Phi \in \bF^{D \times J}$ be a $(d,n)$-ETF with $d < n$ over $\bF$ with  frame bound $A$. Then any $\Psi \in \bF^{D' \times J}$ with $\abs{D'}=n-d$ that satisfies
\[
\Psi^\ast \Psi = A I- \Phi^\ast \Phi
\]
is an $(n-d,n)$-ETF over $\bF$ with frame bound $A$ called a \emph{Naimark complement} of $\Phi$.
\end{prop}
Note that the inner products of distinct vectors of a Naimark complement are simply the negatives of the corresponding vectors in the first frame.  Thus, by Proposition~\ref{prop:switch}, characterizing switching equivalence of an ETF is equivalent to solving the problem for Naimark complements.  Hence, when studying the switching equivalence of ETFs it suffices to study ETFs with $n \geq 2d$.

Given the historical connection of frames to Fourier bases~\cite{DS52}, it is perhaps not surprising that Fourier-related concepts arise in constructions of various ``nice'' frames, like frames in $L^2(\mu)$ for $\mu$ a fractal measure (see, e.g., \cite{dutkay2011beurling}) or Gabor frames for $L^2(\bR^d)$ (see, e.g., \cite{Purple}).  Fourier matrices may also be used to define ``nice'' frames in finite dimensions.
\begin{defn}\label{defn:Fourmat}
Let $G$ be a finite abelian group with character group $\widehat{G} \cong G$.  A \emph{Fourier matrix} of $G$ is a character table of $G$, that is, a matrix in $\bF^{G \times \widehat{G}}$ with $(g,\chi)$-entry $\chi(g)$ for $g \in G$, $\chi \in \widehat{G}$.
\end{defn}
Since $G \cong \widehat{G}$ for finite abelian groups $G$ (and since the formulas for characters that we will use will filter through that isomorphism), we will simply write that the Fourier matrices are elements of $\bF^{G \times G}$.  

By the Schur orthogonality relations, a Fourier matrix for a finite abelian group $G$ with order $n$ is a unitary matrix scaled by $\sqrt{n}$; thus, any $d$-row all-column (i.e., $n$-column) submatrix $\Phi$ satisfies $\Phi \Phi^\ast = n I$.  So, such submatrices always yield tight frames.  Further, character evaluations of abelian groups are always unimodular; thus, such a $\Phi$ will be equal--norm as well.
We can completely characterize when such submatrices yield ETFs using combinatorial design theory (see, e.g., \cite{Handbook}).
\begin{defn}
Let $G$ be a finite abelian group of order $v$ written additively.  A \emph{$(v,k,\lambda)$-difference set} is a subset $D \subseteq G$ of size $k$ such that the multiset of differences contains each non-identity element exactly $\lambda$ times.
\end{defn}
\begin{ex}\label{ex:paleyDS1}
Consider $\{ 1, 2, 4\} \subset \bZ_7$.  Then the difference table is
\[
\begin{array}{c||c|c|c}
- & 1 & 2 & 4 \\ \hline \hline
 1 & 0 & 1 & 3\\ \hline
 2 & 6 & 0 &2 \\ \hline
 4& 4 & 5 & 0
\end{array} \,\,.
\]
That is, the multiset of differences contains each non-identity element of $\bZ_7$ exactly once.  Thus, $\{1, 2, 4\}$ is a $(7,3,1)$-difference set.
\end{ex}
In any abelian group of size $v$, any subset of size $1$, $v-1$, or $v$ trivially forms a difference set. 
(In fact, some authors would preclude calling such sets difference sets due to their degeneracy.)
The difference set in Example~\ref{ex:paleyDS1} is the smallest nontrivial difference set.  In particular, not all abelian groups have a nontrivial difference set.
Examples of the construction of ETFs using difference sets appeared in \cite{StH03,DiFe07,GoRo09}, while the following theorem which completely characterizes the construction may be found in \cite{XZG05}. 
\begin{thm}\label{thm:diffsetcon}
Let $G$ be a finite abelian group with Fourier matrix $F$.  Let $\Phi$ be an all-column submatrix of $F$, with rows indexed by $D$.  Then $\Phi$ is an ETF if and only if $D$ is a difference set in $G$.
\end{thm}

\subsection{The Geometric Structure of ETFs}\label{sec:geomETF}
One way to further understand ETFs is by characterizing their geometric and combinatorial properties via their symmetries and linear dependencies.  We introduce the former concept in this section and the latter in the next section.
\begin{defn}
Let $\Phi=(\varphi_j)_{j \in J} \in \bF^{D \times J}$.  We call $U \in \U(D)$ a \emph{vector automorphism of $\Phi$} if there exists a permutation $\sigma \in S_J$ such that $U\varphi_j = \varphi_{\sigma(j)}$ for $j \in J$.  We call $U \in \U(D)$ a \emph{line automorphism of $\Phi$} if there exists a permutation $\sigma \in S_J$ such that $U \varphi_j \varphi_j^\ast U^\ast = \varphi_{\sigma(j)} \varphi_{\sigma(j)}^\ast$ for $j \in J$.  The subgroup of $S_J$ induced from the vector automorphisms is called the \emph{vector symmetry group $\Sym_v(\Phi)$ of $\Phi$}, and the subgroup of $S_J$ induced from the line automorphisms is called the \emph{line symmetry group $\Sym_{\ell}(\Phi)$ of $\Phi$}.
\end{defn}

\begin{ex}\label{ex:fourorbit}
Let $G$ be a finite abelian group with Fourier matrix $F$ that has a difference set $D$.  Let $\Phi$ be the ETF in $\bF^{D \times G}$ formed from extracting the rows from $F$ corresponding to the difference set. Further, define $M^{\{j\}}\in \U(D)$ for each frame index $j$ as 
\[
\left(M^{\{j\}}\right)_{g,h} = \left(\varphi_j\right)_g \delta_{g,h},
\]
i.e., the diagonal unitary matrix with diagonal entries the entries of $\varphi_j$. Then $\{M^{\{j\}}\}_j$ is a subgroup of $\U(D)$ which is the image of a representation of $G$.  Also, $F$ is the orbit of the all-ones vector $\varphi_0=\mathbbm{1}$ under $\{M^{\{j\}}\}_j$, where $M^{\{j\}}\mathbbm{1} = \varphi_j$.  Since every vector of an ETF is distinct, this action of $G$ on the frame vectors via $\{M^{\{j\}}\}_j$ is free.
This implies that $\Sym_v(\Phi)$ and in turn $\Sym_{\ell}(\Phi)$ each have a subgroup isomorphic to $G$. 
\end{ex}

We will often ask when the automorphism groups have extra structure.
\begin{defn}
Let $k \leq n$. An action of a group $G$ on a non-empty set $X = \set{x_i}{i\in\{1,\hdots,n\}}$  is \emph{$k$-transitive} if for each ordered $k$-tuple with no repeated elements $(x_{i_1}, \hdots, x_{i_k})$ and each permutation $\sigma\in S_n$, there exists a $g \in G$ such that (as ordered tuples)
\[
(g \cdot x_{i_1}, \, g \cdot x_{i_2},\, \hdots,\, g \cdot x_{i_k}) = (x_{\sigma(i_1)}, \, x_{\sigma(i_2)},\, \hdots,\, x_{\sigma(i_k)}). 
\]
The action is \emph{$k$-homogeneous} if for each set of $k$ elements of $X$ with no repeated elements $\{x_{i_1}, \hdots, x_{i_k}\}$ and each permutation $\sigma\in S_n$, there exists a $g \in G$ such that (as sets)
\[
\{g \cdot x_{i_1}, \, g \cdot x_{i_2},\, \hdots,\, g \cdot x_{i_k}\} = \{x_{\sigma(i_1)}, \, x_{\sigma(i_2)},\, \hdots,\, x_{\sigma(i_k)}\}. 
\]
A frame is \emph{$k$-transitive} (respectively, \emph{$k$-homogeneous}) if its line automorphism group is $k$-transitive (respectively, $k$-homogeneous).
\end{defn}
Clearly, $k$-transitivity implies both $(k-1)$-transitivity and $k$-homogeneity. However, the relationship between $k$-homogeneity and $\ell$-homogeneity for $k \neq \ell$ is subtler. Consider a set of size $k$ acted on by the trivial group.  Then the group action is $k$-homogeneous but not $j$-homogeneous for any $1 \leq j < k$. On the other hand, when $k$ is less than or equal to half the size of the set acted on, $k$-homogeneity implies both $(k-1)$-transitivity and $(k-1)$-homogeneity \cite{dixon1996permutation,livingstone1965transitivity} and in many cases also $k$-transitivity \cite{kantor1972khomo}. In Section~\ref{sec:Paley}, we will prove that an infinite class at ETFs have symmetry groups which are $2$-homogeneous but not $2$-transitive.

We make a note of a result from \cite{chien2018projective} (cf.\ Proposition~\ref{prop:switch}) that we will leverage to prove a few results and which was also the basis of the code~\cite{JoeySymm} provided by Joey Iverson that was used to test some ideas.
\begin{thm}\cite{chien2018projective}\label{thm:TP}
Let $\Phi = (\varphi_j)_{j=1}^n$ be an ETF.  Then $\sigma \in \Sym_v(\Phi)$ if and only if for all $j \neq k$
\[
\ip{\varphi_j}{\varphi_k} =\ip{\varphi_{\sigma(j)}}{\varphi_{\sigma(k)}}. 
\]
Similarly,  $\sigma \in \Sym_{\ell}(\Phi)$ if and only if for all $j \neq k \neq l \neq j$
\[
\TP(j,k,\ell) = \TP(\sigma(j),\sigma(k),\sigma(\ell)).
\]
\end{thm}
Since this result relies solely on the inner products, one can see that the vector symmetry groups (respectively, the line symmetry groups) of an ETF and a Naimark complement of an ETF are the same.

\subsection{The Combinatorial Structure of ETFs}\label{sec:combETF}
Understanding the linear dependencies of vectors in ETFs rose to prominence in the search for so-called deterministic restricted isometry property (RIP) measurements in compressed sensing (see, e.g.,~\cite{bandeira2013road}).  However, it was realized that many ETFs have in some sense worst possible linear dependence properties~\cite{FMT12}.
\begin{defn}
Let $\Phi  \in \bF^{D\times J}$ for $\abs{J} > \abs{D}$. The \emph{spark} (or \emph{girth}) of $\Phi$, $\spark(\Phi)$, is the size of the smallest subset of linearly dependent vectors in $\Phi$.
\end{defn}
``Spark'' is the term from compressed sensing literature \cite{DE03}, while ``girth'' is a more classical term from matroid theory (see, e.g., \cite{Ox11}).   RIP concerns subsets up to a certain size of columns of a matrix having condition number close to $1$; so, small subsets of columns being linearly dependent means that the matrix is far from being RIP.
\begin{ex}
Consider an orthonormal basis $(u_j)_{j=1}^d$ for $\bF^d$. Then
\begin{itemize}
\item $\spark\begin{pmatrix} u_1 & \hdots & u_d & \sum_{j=1}^d u_j \end{pmatrix} = d+1$,
\item $\spark\begin{pmatrix} u_1 & \hdots & u_d & u_1 \end{pmatrix} = 2$, and
\item $\spark\begin{pmatrix} u_1 & \hdots & u_d & 0 \end{pmatrix} = 1$,
\end{itemize}
where the singleton set consisting of the zero vector is counted as a linearly dependent set of size $1$, while the rank of each of the vector collections above is $d$.
\end{ex}

Trivially, the largest that the spark of $n> d$ vectors in $\bF^D$ with $\abs{D}=d$ may be is $d+1$, which happens when any subset of $d$ vectors is linearly independent.  When this happens, we say the vectors are \emph{full spark}.  Full spark frames (respectively, full spark tight frames) of $n$ vectors in $\bF^D$ form an open dense set in the space of frames (respectively, tight frames) of $n$ vectors in $\bF^D$~\cite{CahMS,ACM12,CMS13}.  However, the same cannot be said when restricting to ETFs. Applying a now classical result from the compressed sensing literature~\cite{DE03} to ETFs, we have the following result.
\begin{prop}
Let $\Phi$ be a $(d,n)$-ETF with $n> d>1$.  Then
\[
\spark(\Phi) \geq 1 + \sqrt{\frac{d(n-1)}{n-d}}.
\]
\end{prop}
If $n > d+1$, then this worst case bound is strictly less than $d+1$; that is, any such ETFs which saturate this bound are not full spark. The infinite class of ETFs constructed in~\cite{FMT12} all saturate this worst-case bound for spark.  Inspired by this, the author of this paper and coauthors characterized all ETFs which saturate this bound~\cite{FJKM17}.
\begin{thm}
Let $\Phi $ be a $(d,n)$-ETF with $n> d>1$. Then 
\[
\spark(\Phi) = 1 + \sqrt{\frac{d(n-1)}{n-d}}
\]
if and only if $\Phi$ contains a subset of vectors which form a simplex ETF for their span.
\end{thm}
This led to the following definition \cite{FJKM17}.
\begin{defn}
Let $\Phi$ be an ETF. The \emph{binder} of $\Phi$ is the set of subsets of $\Phi$ (or their corresponding indices) which form simplex ETFs for their span.
\end{defn}
If the binder is non-empty, then each of the elements of the binder (which are subsets of $\Phi$) have the same size. In certain cases (e.g., Gabor-Steiner frames~\cite{BoKi18}), the binder forms a type of combinatorial design call a BIBD. For more on combinatorial designs see, e.g.,~\cite{Handbook}.
\begin{defn}
A \emph{$t$-$(v,k,\lambda)$-design} or \emph{block design} is an ordered pair $(V,\mathcal B)$, where $V$ is a set of size $v$ and $\cB$ is a collection of subsets of $V$ called \emph{blocks} such that each block $\beta \in \mathcal B$ has the same size $k$ and any subset of $t$ distinct elements in $V$ is contained in precisely $\lambda>0$ blocks.  When $t=2$, the design is also called a \emph{balanced incomplete block design (BIBD)}.
\end{defn}
Note that $t$-designs are also $s$-designs for any $1 \leq s \leq t$.
\begin{ex}\label{ex:DSBIBD}
Consider the following subsets of $\bZ_7$ which result from shifting the difference set in Example~\ref{ex:paleyDS1} by each element of the group:
\[
\left\{ \{1,2,4\}, \,  \{2,3,5\} ,\,  \{3,4,6\}, \,  \{0,4,5\} ,\,  \{1,5,6\}, \,  \{0,2,6\}, \,  \{0,1,3\} \right\}.
\]
Note that $\{1,2\}$ is only in the first block, $\{1,3\}$ only in the seventh, $\{1,4\}$ only in the first, and so on. This is a $2$-$(7,3,1)$-design a.k.a.\ a $(7,3,1)$-BIBD.  
\end{ex}
In practice we find that many ETFs have an empty binder. This does not mean that the ETF is full spark, only that it does not have worst case spark.  So, one may ask to study the properties of the smallest dependent sets.  This leads us to introduce a new definition.
\begin{defn}\label{defn:bnder}
Let $\Phi$ be an ETF with spark $s$. The \emph{bender}\footnote{``Bender'' is in honor of the smack-talking robot on Futurama, a show which is filled with math jokes.} of $\Phi$ is the set of subsets of $\Phi$ (or their corresponding indices) which are size $s$ and linearly dependent.  The elements of the bender are called \emph{short circuits}.
\end{defn}
In matroid theory, minimally dependent subsets of vectors, i.e., those for which any proper subset is independent, are called \emph{circuits}.  The size of the smallest circuit(s) is the girth, i.e., the spark. Hence the name ``short circuit.''\footnote{Thanks to Matt Fickus, who suggested this name.}   

If $\Phi$ has a nontrivial binder, then the bender is the binder and the short circuits are all simplex ETFs for their span.  On the other extreme, if $\Phi$ is full spark $d+1$, then the bender is the set of all subsets of size $d+1$ of the vectors.

It was noted in  \cite{fickus2023doubly} that the binder of an ETF with a $2$-transitive line automorphism group is a BIBD.  However, folklore in combinatorial design theory (see, e.g., \cite[Remark 4.29]{Handbook}) means a much more general result holds.
\begin{prop}\label{prop:benderdesign} 
Let $\Phi$ be an ETF with $k$-homogeneous line symmetry group and spark $\geq k$.  Then the bender forms a $k$-design.
\end{prop}
\begin{proof}
Note that neither scaling vectors by unimodular scalars nor performing a unitary change of basis affects linear dependency of a set. So, we will consider the short circuits to consist of the sets of the lines spanned by each of the vectors, and, thus, acting on a short circuit with a line automorphism results in a short circuit.  For any two different sets $X_1$ and $X_2$ of $k$ frame vectors, there is a least one element of the line symmetry group mapping $X_1$ to $X_2$, meaning that the size of the set of short circuits containing $X_1$ is the same as the size of the set of short circuits containing $X_2$.  Thus, the bender is a $k$-design.
\end{proof}

\section{Paley ETFs}\label{sec:Paley}
We will apply the tools from Sections~\ref{sec:geomETF} and~\ref{sec:combETF} to combinatorially and geometrically characterize one infinite class of ETFs, the Paley (difference set) ETFs.  
These ETFs are generated using the structure of finite fields.  
Thus, we begin in Section~\ref{sec:ff} by reviewing some basic facts about finite fields. 
For each prime power $q \equiv 3 \bmod 4$, there is a Paley ETF $\Phi_q$.  
We will prove in Section~\ref{sec:paleyprim} that $\Sym_v(\Phi_q)=\Sym_{\ell}(\Phi_q)$ act $2$-homogeneously but not $2$-transitively on the vectors and the lines, respectively (Theorem~\ref{thm:Paleydbhm}).  
The Paley ETFs thus have a bender which is a BIBD.  
Then in Section~\ref{sec:altproof}, we provide an alternate proof of Theorem~\ref{thm:Paleydbhm} that filters through Theorem~\ref{thm:TP}. 
This proof does not yield the explicit vector automorphisms but yields techniques which may be useful for other classes of ETFs. 
A quick example of the technique being applied to another infinite class of ETFs (Gabor-Steiner) is also given. 

\subsection{Finite Fields}\label{sec:ff}
We begin by reviewing some basic facts about finite fields (see, e.g.,~\cite{DummF,lidl1997finite}). 
If $q = p^s$ for some prime $p$ and $\bF_q$ is the finite field of size $q$, then $\bF_q^\times = \bF_q\backslash \{0\}$ is a cyclic group under multiplication.  
We let $\operatorname{Aut}(\bF_q)$ denote the group of Galois field automorphisms of $\bF_q$, which is of order $s$.

For fixed $q=p^s$, let $\bF_p$ denote the smallest subfield of $\bF_q$, i.e., $\bF_p \cong \{0, 1, \hdots, p-1\} \subseteq \bF_q$.
Let $\tr: \bF_q \rar \bF_p$ be the field trace
\[
\tr(x) = x + x^p + x^{p^2} + \cdots + x^{p^{s-1}},
\]
which is invariant under field automorphisms.
Fix a primitive $p$th root of unity $\zeta$.
Then each character of the additive group on $\bF_q$ is defined by $m \in \bF_q$ as
\[
\chi_m(x) = \zeta^{\tr(mx)}, \quad x \in \bF_q.
\]
This means that the Fourier matrix of the additive group on $\bF_q$ may be represented as $F \in \bC^{\bF_q \times \bF_q}$ with entries
\[
F_{x,m} = \chi_m(x) =  \zeta^{\tr(mx)}, \quad x,m \in \bF_q.
\]

We denote the quadratic residues and non-residues of $\bF_q$ as
\begin{align*}
\QR(\bF_q) &= \set{x^2}{x \in \bF_q^\times} \\
\NR(\bF_q) &= \bF_q^\times \backslash \QR.
\end{align*}
Thus, $\bF_q = \QR(\bF_q) \cup \NR(\bF_q) \cup \{0\}$.  
When the underlying field is clear from context, we will often simply write $\QR$ and $\NR$.  
Note that $\QR$ is thus a subgroup of the multiplicative cyclic group $\bF_q^\times$ of index $2$, meaning it has size $(q-1)/2$ and is also cyclic.  
\begin{ex}
We have that $\QR(\bF_7)= \{ 1, 2, 4\}$ since in $\bF_7$
\[
1^1 = 1, \quad 2^2 = 4, \quad 3^2 = 2, \quad 4^2 = 2, \quad 5^2 = 4, \quad 6^2 = 1,
\]
i.e., precisely the difference set in $\bZ_7$ (i.e., the additive group of $\bF_7$) in Example~\ref{ex:paleyDS1}.  
\end{ex}

When $q \equiv 3 \bmod 4$, $\NR = -\QR$; that is, the quadratic non-residues are precisely the negations of the quadratic residues. This fact implies the following (see, e.g.~\cite{Handbook}).
\begin{thm} \label{thm:paleydiff}
When $q \equiv 3 \bmod 4$, $\QR(\bF_q)$ forms a $(q,(q-1)/2,(q-3)/4)$-difference set in the additive group of $\bF_q$.
\end{thm}

Thus, it follows from Theorem~\ref{thm:diffsetcon} that if $q \equiv 3 \bmod 4$, $F$ is the Fourier transform of the additive group of $\bF_q$, and $\Phi$ is the $(q-1)/2 \times q$ submatrix of  $F$ consisting of the rows indexed by $\QR(\bF_q)$, then $\Phi$ must be an ETF, which we call a Paley ETF. 
\begin{defn}
If $q \equiv 3 \bmod 4$, the \emph{Paley ETF} $\Phi_q $ is given as $\begin{pmatrix} \varphi_x\end{pmatrix}_{x \in \bF_q}$ with each $\varphi_x \in \bC^{\QR(\bF_q)}$ defined as
\[
\left(\varphi_x\right)_m = \chi_m(x) = \zeta^{\tr(mx)}, \quad m \in \QR(\bF_q).
\]
\end{defn}

We note that some authors (see, e.g., \cite{bandeira2013road}) use the terminology ``Paley ETF'' to refer to $(d,2d)$-ETFs which are also constructed leveraging number theory about quadratic residues.

\begin{ex}\label{ex:paley7}
The smallest nontrivial Paley ETF is $\Phi_7$.  Namely, let $\zeta$ be a primitive $7$th root of unity.  Then
\[
\Phi_7 = \begin{pmatrix} 1 & \zeta & \zeta^2 & \zeta^3 & \zeta^4 & \zeta^5 & \zeta^6\\1 & \zeta^2 & \zeta^4 & \zeta^6 & \zeta & \zeta^3 & \zeta^5\\1 & \zeta^4 & \zeta & \zeta^5 & \zeta^2 & \zeta^6 & \zeta^3\end{pmatrix}
\]
is one ordering of the $(3,7)$-ETF which is the Paley ETF for $\bF_7$.  Note that $\QR(\bF_7)= \{ 1, 2, 4\}$ and thus $\NR(\bF_7) = \{ 3, 5, 6\}$.  Further, other than the $0$th column, half of the columns have exponents that are $\QR$ and half have $\NR$.  Finally, note that (cf.\ Example~\ref{ex:fourorbit})
\[
\Phi_7 = \set{M^k\mathbbm{1}}{k \in \{ 0, \,\hdots,\, 6\}}, \quad \textrm{where} \quad M= \begin{pmatrix} \zeta & 0 & 0 \\ 0 & \zeta^{2} & 0 \\ 0 & 0 & \zeta^{4}\end{pmatrix} \quad \mathbbm{1} = \begin{pmatrix} 1 \\ 1 \\1 \end{pmatrix}.
\]
\end{ex}

\subsection{Geometric and Combinatorial Structure of Paley ETFs}\label{sec:paleyprim}
In this section, we will explicitly give the vector automorphisms of $\Phi_q$ for $q=p^s>3$ a prime power with $q \equiv 3 \bmod 4$ and show that the vector and line symmetry groups act $2$-homogeneously but not $2$-transitively.  
The line automorphism group is generated by the following operators (with unimodular scalars).

\begin{defn}\label{defn:opers}
Let $q=p^s>3$ be a prime power with $q \equiv 3 \bmod 4$. 
For each $b \in \bF_q$, $m \in \QR(\bF_q)$, and $\sigma \in \operatorname{Aut}(\bF_q)$, we define the following matrices in $\U(\QR(\bF_q))$,
\[
\left( M^{\{b\}} \right)_{x,y} = \chi_x(b)\delta_{x,y}= \zeta^{\tr(xb)}\delta_{x,y}, \quad \left(T^{\{m\}}\right)_{x,y} = \delta_{mx,y}, \quad \textrm{and} \quad  \left( \Pi^{\{\sigma\}} \right)_{x,y} = \delta_{x^{\sigma^{-1}},y}.
\]
\end{defn}

As in Example~\ref{ex:paley7}, we may characterize Paley ETFs as the orbit of $\mathbbm{1}$ under $M^{\{b\}}$ for $b \in \bF_q$.
Note that the $M^{\{b\}}$ are modulations forming a group of order $q$, the $T^{\{m\}}$ are permutations forming a group of order $(q-1)/2$, and the $\Pi^{\{\sigma\}}$ permutation matrices forming a group of order $s$.
This in turn is why we view $T^{\{m\}}$ as translation matrices denoted with $T$.

We have one final lemma before proving the main theorem.
\begin{lem}\label{lem:AGamL}
Let $q = p^s>3$ be a prime power with $q \equiv 3 \bmod 4$. Assume that $G$ acts faithfully, transitively, and $2$-homogeneously but not $2$-transitively on $\bF_q$. Assume further that
\[
\tilde{G} \cong\set{f \in \operatorname{A\Gamma L}(\bF_q)}{f(x) = mx^\sigma+b, \,\, b \in \bF_q,\,\, m \in \QR(\bF_q),\,\,\sigma \in \operatorname{Aut}(\bF_q)}
\]
is a subgroup of $G$, where 
\[
\operatorname{A\Gamma L}(\bF_q) = \set{f: \bF_q \rar \bF_q}{f(x) = m x^\sigma + b,\,\,b \in \bF_q, \,\, m \in \bF_q^\times, \,\,\sigma \in \operatorname{Aut}(\bF_q) }.
\] 
Then $G = \tilde{G}$.
\end{lem}
\begin{proof}
Theorem 1 of~\cite{kantor1972khomo} yields that any group which acts $2$-homogeneously but not $2$-transitively on a set of size at least $4$ must be (permutation equivalent to) a subgroup of $ \operatorname{A\Gamma L}(\bF_q)$.  Note that since $q \equiv 3 \bmod 4$, $s$ and the cardinality of $\QR(\bF_q)$ must also both be odd, i.e., $\tilde{G}$ has odd order.  Also, since $\QR(\bF_q)$ is index $2$ in $\bF_q^\times$, $G$ is a maximal subgroup of $\operatorname{A\Gamma L}(\bF_q)$ of odd order.  Note that $\operatorname{A\Gamma L}(\bF_q)$ acts $2$-transitively on $\bF_q$ and that $\tilde{G}$ is a maximal proper subgroup. Hence, $G = \tilde{G}$, as desired.
\end{proof}

\begin{thm}\label{thm:Paleydbhm}
Let $q=p^s>3$ be a prime power with $q \equiv 3 \bmod 4$.    
Then
\begin{align*}
\lefteqn{G:= \set{M^{\{b\}} T^{\{m\}} \Pi^{\{\sigma\}} }{b \in\bF_q, \, \, m\in \QR(\bF_q),\,\, \sigma \in \operatorname{Aut}(\bF_q)   }}\\
 &\cong  \set{f \in \operatorname{A\Gamma L}(\bF_q)}{f(x) = mx^\sigma+b, \,\, b \in \bF_q,\,\, m \in \QR(\bF_q),\,\,\sigma \in \operatorname{Aut}(\bF_q)}.
\end{align*}
is the group of vector automorphisms of the Paley ETF $\Phi_q$. The vector and line symmetry groups are $\Sym_v(\Phi_q) = \Sym_{\ell}(\Phi_q) \cong G$. The \emph{symmetry groups act} $2$-homogeneously but not $2$-transitively on the vectors and the lines.

\end{thm}
A note of warning for those accustomed to thinking of an affine group as a translation group semidirect product a scaling group, in particular those with a background in wavelet theory: The group of matrices $\{T^{\{m\}}\}$ which translate the entries of the vectors act like dilations in the correspondence to $\operatorname{A\Gamma L}(\bF_q)$.
\begin{proof}
We first verify that $M^{\{b\}} T^{\{m\}} \Pi^{\{\sigma\}} \varphi_x =\varphi_{mx^\sigma + b}$ for all $x \in \bF_q$.
To that end, we compute
\begin{align*}
M^{\{b\}} T^{\{m\}} \Pi^{\{\sigma\}} \varphi_x &= M^{\{b\}} T^{\{m\}} \Pi^{\{\sigma\}} ( \zeta^{\tr(yx)})_y = M^{\{b\}} T^{\{m\}} ( \zeta^{\tr(y^{\sigma^{-1}} x)})_y\\
&= M^{\{b\}} T^{\{m\}} ( \zeta^{\tr(y x^{\sigma})})_y = M^{\{b\}} ( \zeta^{\tr(y mx^{\sigma})})_y =( \zeta^{\tr(y (mx^{\sigma}+b))})_y \\
&= \varphi_{mx^\sigma + b}.
\end{align*}
Thus, $G$ faithfully permutes the vectors of $\Phi_q$ as desired and is isomorphic to the given subgroup of $\operatorname{A\Gamma L}(\bF_q)$.

Now we would like to show that $\Sym_v(\Phi_q) = \Sym_{\ell}(\Phi_q) \cong G$ and that the groups act $2$-homogeneously but not $2$-transitively on the vectors/lines. 
Note that the vector symmetry group $\Sym_v(\Phi_q)$ and thus also the line symmetry group $\Sym_{\ell}(\Phi_q)$ act transitively since the vectors in $\Phi_q$ are the orbit of $\mathbbm{1}$ under $\{M^{\{b\}}\}$. 
We further have that $G$ acts $2$-homogeneously on the vectors as follows.  
Consider $x \neq y$ and $z \neq w$ in $\bF_q$. 
Then either $x-y$ or $y-x$ is a quadratic residue, similarly for $(z,w)$. 
Assume without loss of generality that $x-y, z-w \in \QR(\bF_q)$.
Then $(x-y)/(z-w) \in \QR(\bF_q)$ so that
\[
M^{\{y\}} T^{\{(x-y)/(z-w)\}}  M^{\{-w\}} \varphi_z = \varphi_x \quad \textrm{and} \quad M^{\{y\}} T^{\{(x-y)/(z-w)\}}  M^{\{-w\}} \varphi_w = \varphi_y; 
\]
i.e., $M^{\{y\}} T^{\{(x-y)/(z-w)\}}  M^{\{-w\}}$ maps $\{\varphi_z,\varphi_w\}$ to $\{\varphi_x,\varphi_y\}$. 
Hence, $\Sym_v(\Phi_q)  \leq  \Sym_{\ell}(\Phi_q)$ act $2$-homogeneously on the vectors and lines, respectively.

We know from the complete classification of all $2$-transitive ETFs~\cite{iverson2024doubly} that $\Sym_{\ell}(\Phi_q)$ is not $2$-transitive and thus neither is $\Sym_v(\Phi_q)$. It thus follows from Lemma~\ref{lem:AGamL} that $\Sym_v(\Phi_q) = \Sym_{\ell}(\Phi_q) \cong G$.
\end{proof}

This theorem should not be incredibly surprising, given that Paley designs are the only $2$-homogeneous but not $2$-transitive Hadamard designs~\cite{kantor1969auto}. 
Note, however, that the Paley Hadamard designs built from $\bF_{11}$ and $\bF_7$ are $2$-transitive~\cite{todd1933combinatorial,kantor1969transitive}, differentiating this result from the similar result for Paley ETFs. 
As we will see in Theorem~\ref{thm:khomETF}, Paley skew-conference matrices, which are also built using quadratic residues, also yield an infinite class of ETFs which are $3$-homogeneous but not $3$-transitive.  

We can now use the symmetry to understand the combinatorial structure of the Paley ETFs.

 \begin{cor}\label{cor:BIBD}
 Let $q>3$ be a prime power with $q \equiv 3 \bmod 4$. The bender of the Paley ETF $\Phi_q$ is a BIBD.
 \end{cor}
 \begin{proof}
 This follows immediately from Theorem~\ref{thm:Paleydbhm} and Proposition~\ref{prop:benderdesign} 
 \end{proof}
 
Due to a result from Chebotar\"ev~\cite{stevenhaben1996chebat}, the Paley ETFs $\Phi_p$ with $p$ prime are full spark, i.e., have spark $(p-1)/2 +1 = (p+1)/2$.  Their bender is just trivially the set of all subsets of size $(p+1)/2$, which is a $t$-design for all $1 \leq t \leq (p+1)/2$.  For a discussion of $p^s$ with $s >1$, see Question~\ref{quest:bender}.

Thus, we can now finish the story of our throughline example.
\begin{ex}
Consider $\Phi_7$, parameterized as in Example~\ref{ex:paley7}.  Then 
\[
\operatorname{Sym}_v(\Phi_7) = \operatorname{Sym}_{\ell}(\Phi_7) \cong \bZ_7 \rtimes \bZ_3,
\]
with line automorphism group generated by unimodular scalars and 
\[
\left\{ \begin{pmatrix} 1 & 0 & 0 \\ 0 & \zeta^2 & 0 \\ 0 & 0 & \zeta^4\end{pmatrix}^j \begin{pmatrix} 0 & 0 & 1 \\ 1 & 0 & 0 \\ 0 & 1 & 0 \end{pmatrix}^k \, \Bigg \vert \, j \in\{0, 1, \hdots, 6\}, k \in \{0, 1 ,2\}\right\}.
\]
Finally, the bender of $\Phi_7$ is the binder, which is all subsets of size $4$.
\end{ex}
\subsection{Alternate Proof of Theorem~\ref{thm:Paleydbhm}} \label{sec:altproof}
In this section, we will provide an alternate proof characterizing the symmetry groups of $\Phi_q$ which filters through the inner products, respectively, triple products (i.e., uses Theorem~\ref{thm:TP}). This has the benefit that some of the techniques might be useful in the analysis of symmetry groups of other ETFs but the drawback that the explicit unitary matrices which are vector, respectively, line automorphisms are not constructed. We conclude with a short example of using the approach on another class of ETFs (Gabor-Steiner).  In order to use Theorem~\ref{thm:TP}, we need to characterize the inner products and triple products of the Paley ETF vectors. 

\begin{lem}\label{lem:paleyTP}
Let $q=p^s>3$ be a prime power with $q \equiv 3 \bmod 4$ and consider the Paley ETF $\Phi_q$. 
Further let $\zeta$ be the primitive $p$th root of unity used to produce $\Phi_q$.  
Finally, for any $m \in \QR(\bF_q)$, set
\[
a = \sum_{\ell \in \QR(\bF_q)} \zeta^{\tr(\ell m)}.
\]
Then for pairwise distinct $j,k,\ell$, we have
\[
\TP(j,k,\ell) =\left\{ \begin{array}{lr} a^3; & j-k, k-\ell, \ell - j \in \QR\\[2pt] \abs{a}^2 a; & \textrm{ two of $j-k, k-\ell, \ell-j \in \QR$} \\[2pt] \abs{a}^2 \overline{a}; & \textrm{ one of $j-k, k-\ell, \ell-j \in \QR$} \\[2pt]  \overline{a}^3; & \ j-k, k-\ell, \ell - j \in \NR \end{array}\right..
\]
\end{lem}
\begin{proof}
First, we note that since $\QR(\bF_q)$ is a group under multiplication, $a$ is well-defined and invariant to choice of $m\in \QR(\bF_q)$.
Now, we note that
\[
\ip{\varphi_j}{\varphi_k} = \sum_{\ell\in \QR(\bF_q)}\zeta^{\tr(\ell(j-k))}=\left\{ \begin{array}{lr} (p-1)/2 & j=k \\ a & j-k \in \QR \\\overline{a}& j-k \in \NR\end{array}\right.
\]
since $\ell(j-k) \in \QR$ if and only if $j-k \in \QR$ and also $\NR=-\QR$. 
So, for distinct $j, k, \ell$, the desired result immediately follows.
\end{proof}
These triple products are distinct.
\begin{lem}\label{lem:TPsdiff}
Let $a$ be as in Lemma~\ref{lem:paleyTP}. Then $a^3$, $\abs{a}^2 a$, $ \abs{a}^2 \overline{a}$, and $\overline{a}^3$ are pairwise distinct.
\end{lem}
\begin{proof} At least one pair of the numbers in the above list is equal if and only if $a \in \{\overline{a}, -\overline{a}, \overline{a}\eta, \overline{a}\eta^2\}$ for $\eta$ a primitive third root of unity. Let $\omega$ be any sixth root of unity. Note that $a$ and $\overline{a}$  are sums of $(p-1)/2$ distinct primitive $p$th roots of unity (indexed by the quadratic residues and quadratic non-residues, respectively).  Thus, $a - \overline{a}\omega$ is a rational linear combination of $p-1$ distinct $6p$th roots of unity, where the smallest order of the ratio of any two of the roots of unity is $p$ (since $6$ and $p$ are relatively prime). Theorem 2.6 in~\cite{lenstra1979vanishing} (see also~\cite{conway1976trig}) yields that a rational linear combination of roots of unity, where the smallest order of the ratio of any two of the roots of unity is a prime $p$, can never vanish if there are fewer than $p$ terms. Thus, $a - \overline{a}\omega$ cannot be $0$, yielding that $a^3$, $\abs{a}^2 a$, $ \abs{a}^2 \overline{a}$ are four distinct numbers. \end{proof}

We are now ready to give an alternate proof of Theorem~\ref{thm:Paleydbhm}.
\begin{proof}[Alternate Proof of Theorem~\ref{thm:Paleydbhm}]
It follows from Lemma~\ref{lem:paleyTP} that the value of the inner product of two distinct vectors in $\Phi_q$ depends solely on whether the difference of the indices is a quadratic residue. 
In particular, if we cyclically shift all of the indices by $\ell \in \bF_q$, none of the inner products or triple products change (as $(j-\ell) - (k-\ell) = j-\ell$).  
This means that $\Sym_v(\Phi_q)$ acts transitively. 
Also, note that multiplying the indices by a quadratic residue does not change the inner products or triple products since $\QR$ is a subgroup of $\bF_q^{\times}$/

Let $x \neq y$ and $z \neq w$ be in $\bF_q$. 
We may choose without loss of generality an ordering such that $x-y, z-w \in \QR$.  
Then there is a quadratic residue $m$ that scales $x-y$ to $z-w$. So, we apply the following sequence of permutations: 
\[
(x,y) \,\rar\, (x-y, y-y)=(x-y,0) \,\rar \,(m(x-y),m(0)) = (z-w, 0) \, \rar \, ((z-w)+w,0+w) = (z,w).
\]
Hence, $\Sym_v(\Phi_q)$ acts $2$-homogeneously.   
We know from~\cite{iverson2024doubly} that $\Sym_{\ell}(\Phi_q)$ is not $2$-transitive and thus neither is $\Sym_v(\Phi_q)$. 
The desired result now follows from Lemma~\ref{lem:AGamL}.
\end{proof}

As an another example of this technique, we consider the \emph{Gabor-Steiner ETFs} \cite{BoKi18}. (See also related constructions in \cite{BoEl10a,BoEl10b,IJM17,FJMPW17}.)  Namely, each finite abelian group $G$ of odd order $m$ generates a so-called Gabor-Steiner ETF $\Phi_G$ of $m^2$ vectors in $\bC^{m(m-1)/2}$.  The vectors of $\Phi_G$ are the orbit of a projective unitary representation of $G \times G$ and are thus parameterized by $G \times G$.  For $G = (\bF_p, +)$ with $p$ an odd prime and $\zeta$ a primitive $p$th root of unity, Lemma 5.2 of~\cite{BoKi18} yields that the triple products of the vectors are 
\[
\TP((k,\kappa),(\tilde{k},\tilde{\kappa}),(\widehat{k},\widehat{\kappa})) = \zeta^{[(k\widehat{\kappa}-\widehat{k}\kappa)+(\tilde{k}\kappa-k\tilde{\kappa})+(\widehat{k}\tilde{\kappa}-\tilde{k}\widehat{\kappa})](p+1)/2}.
\]
One can see from inspection (i.e., just writing out the coordinates and simplifying) that for $\operatorname{Sp}(2s,p) \subset \bF_p^{2s \times 2s}$ the group of $2s \times 2s$ symplectic matrices,
\[
\operatorname{ASp}(2,p) = \set{f: \bF_p^2 \rar \bF_p^2}{f(x) = M x + b, \,\, M \in \operatorname{Sp}(2,p), \,\, b \in \bF_p^2} \leq \Sym_{\ell}(\Phi_{ (\bF_p, +)})
\]
since those maps preserve $(k\widehat{\kappa}-\widehat{k}\kappa)+(\tilde{k}\kappa-k\tilde{\kappa})+(\widehat{k}\tilde{\kappa}-\tilde{k}\widehat{\kappa})$ for $(k,\kappa),(\tilde{k},\tilde{\kappa}),(\widehat{k},\widehat{\kappa}) \in \bF_p^2$.  In fact, one may further show that $\operatorname{ASp}(2s,p)$ preserves the triple products of the Gabor-Steiner ETFs for $G = (\bF_p^s, +)$ with $s \geq 1$. Since $\operatorname{ASp}(2s,p)$ acts $2$-transitively on $\bF_p^{2s}$, this means that the Gabor-Steiner ETFs for $G = (\bF_p^s, +)$ have $2$-transitive line symmetry group.  In fact, it is shown in~\cite{dempwolff2023trans} that for all odd primes $p$ and all $s \geq 1$, there are ETFs of $p^{2s}$ vectors in $\bC^{p^s(p^s-1)/2}$ with $2$-transitive line symmetry group isomorphic to $\operatorname{ASp}(2,p)$.  These ETFs are the unique $2$-transitive ETFs of these parameters up to isomorphism.  In particular, this means that the line symmetry groups of $\Phi_{ (\bF_p^s, +)}$ must be isomorphic to $\operatorname{ASp}(2s,p)$. The proofs in~\cite{dempwolff2023trans} rely on the classification of $2$-transitive groups and representation theory, but we were able use to the techniques of this section to quickly show that the line symmetry group of each $\Phi_{ (\bF_p^s, +)}$ acts $2$-transitively and contains a subgroup isomorphic to $\operatorname{ASp}(2s,p)$.

\section{$k$-Homogeneous ETFs}\label{sec:homoETF}

In this section, we characterize $k$-homogeneous ETFs.   The following result is generally known, but we include a proof here for completeness.
\begin{lem}\label{lem:sn}
The line symmetry group of trivial ETFs with $n$ vectors is $S_n$.  
If an ETF is a regular simplex ETF, the Naimark complement of a regular simplex ETF, or an orthonormal basis, then the vector symmetry group is also isomorphic to  $S_n$.
\end{lem}
\begin{proof}
Given any two ordered orthonormal bases $U_1$ and $U_2$ of $\bC^n$, $U_2 U_1^{-1}$ is a unitary matrix that maps the basis $U_1$ to the basis $U_2$.  
Thus, the automorphism group of the vectors and in turn the lines is $S_n$. 
A $(d,n)$-ETF has $n=d+1$ vectors if and only if all of the triple products are negative \cite{FJKM17}; similarly, an ETF with $d=1$ is a Naimark complement of a simplex ETF and thus has all positive triple products.  
Hence, by Theorem~\ref{thm:TP}, any permutation of the lines is in the symmetry group. 
Finally, if the ETF is a regular simplex ETF (respectively, a Naimark complement of a regular simplex ETF), then the inner products are all negative (respectively, positive) and thus any permutation of the vectors is possible (see, e.g.,~\cite{ChWa16}).
\end{proof}
We also include the following converse.
\begin{lem}\label{lem:trip}
No nontrivial ETFs have $3$-transitive line symmetry groups. 
The only ETFs with $2$-transitive vector symmetry groups are orthonormal bases, regular simplex ETFs, and Naimark complements of regular simplex ETFs.
\end{lem}
\begin{proof}
The first result is Theorem II.6 in \cite{King19}. The second comes from the fact that if the vector symmetry group is $2$-transitive, then any inner product of distinct vectors is mapped to any other inner product. 
\end{proof}
Note that this gives an alternate proof that the vector symmetry group of the prime order Paley ETFs $\Phi_p$ is not $2$-transitive. We note that the following immediately follows from the definition of homogeneity, but we record it here as we will use it frequently.
\begin{lem}\label{lem:homcomp}
If a group acts $k$-homogeneously on a set of size $n$, then it also acts $(n-k)$-homogeneously.
\end{lem}
We now have a final lemma before stating and proving the main result of this section.
\begin{lem}\label{lem:3hom}
If $\Phi$ is a $(d,n)$-ETF with $n > 3$ vectors which has $3$-homogeneous line symmetry group, then up to switching equivalence, the inner products are all $0$ ($n=d$), negative ($n=d+1$), positive ($d=1$), or purely imaginary ($n=2d$).
\end{lem}
\begin{proof}
Let $\Phi = (\varphi_j)_j$ be arbitrary and let $j,k,\ell$ be indices.  Note that
\begin{align*}
\overline{\TP(j,k,\ell)}=\overline{\ip{\varphi_j}{\varphi_k}\ip{\varphi_k}{\varphi_{\ell}}\ip{\varphi_{\ell}}{\varphi_j}} &= \ip{\varphi_k}{\varphi_j}\ip{\varphi_{\ell}}{\varphi_k}\ip{\varphi_j}{\varphi_{\ell}} = \ip{\varphi_j}{\varphi_{\ell}}\ip{\varphi_{\ell}}{\varphi_k}\ip{\varphi_k}{\varphi_j}=\TP(j,\ell,k).
\end{align*}
Also note that triple products are invariant to cyclic permutation of the indices; i.e.,
\[
\TP(j,k,\ell) = \TP(\ell,j,k) = \TP(k,\ell, j).
\]
Thus, for any fixed set of indices $\{j,k,\ell\}$, there are only two possible triple products regardless of order, $\TP(j,k,\ell)$ and $\overline{\TP(j,k,\ell)}$.
Now assume that $\Phi$ is an ETF with $n > 3$ vectors which has a $3$-homogeneous line symmetry group. Since $\Phi$ is an ETF, every triple product (of distinct indices) has modulus $\abs{a}^3$ for some $a \in \bC$. And further, since $\Phi$ is $3$-homogeneous and every set of three lines is mapped to every other set of three lines, $\TP(j,k,\ell) \in \{\TP(\tilde{j},\tilde{k},\tilde{\ell}),\, \overline{\TP(\tilde{j},\tilde{k},\tilde{\ell})}\}$ for any distinct triples of distinct indices.  That is, every triple product must be of the form $a^3$  or $\overline{a^3}$. Assume without loss of generality that the norms of the vectors of the ETF are all $1$. An ETF with unit norm vectors such that the spectral norm of 
\[
\begin{pmatrix} 1 & \ip{\varphi_j}{\varphi_k} &  \ip{\varphi_j}{\varphi_{\ell}} \\
 \ip{\varphi_k}{\varphi_j} & 1& \ip{\varphi_k}{\varphi_{\ell}} \\
  \ip{\varphi_{\ell}}{\varphi_j} & \ip{\varphi_{\ell}}{\varphi_k}&1 
\end{pmatrix} 
\]
is constant for all $j \neq k \neq \ell \neq j$ is called \emph{$3$-uniform}~\cite{BP05} (or \emph{$3_c$-uniform}~\cite{hoffman2012complex}).  It was shown in~\cite{hoffman2012complex} that the only $3_c$-uniform ETFs with $d>1$ are, up to switching equivalence, orthonormal bases, simplices, or $(d,2d)$-ETFs with purely imaginary inner products.  Thus, we would like to show that $3$-homogeneous ETFs are $3_c$-uniform. To that end, we compute the characteristic polynomial of the Gram submatrix above:
\begin{align*}
\lefteqn{\det\begin{pmatrix} 1-\lambda & \ip{\varphi_j}{\varphi_k} &  \ip{\varphi_j}{\varphi_{\ell}} \\
 \ip{\varphi_k}{\varphi_j} & 1-\lambda& \ip{\varphi_k}{\varphi_{\ell}} \\
  \ip{\varphi_{\ell}}{\varphi_j} & \ip{\varphi_{\ell}}{\varphi_k}&1-\lambda 
\end{pmatrix}}\\
 &= (1-\lambda)^3 -(1-\lambda)\left(\absip{\varphi_j}{\varphi_k}^2 +\absip{\varphi_j}{\varphi_{\ell}}^2  + \absip{\varphi_k}{\varphi_{\ell}}^2\right) + \TP(j,k,\ell) +  \overline{\TP(j,k,\ell)} \\
&= (1-\lambda)^3 -3(1-\lambda) \abs{a}^2 + 2\operatorname{Re}  a^3,
\end{align*}
which is constant for any distinct $j,k,\ell$.  Thus, the spectral norms of any of the $3 \times 3$ Gram submatrices are equal to each other as their spectra are all equal.
\end{proof}
The concept of $k$-uniformity was introduced to study robustness of frames to $k$ erasures~\cite{HoPa04,BP05}. (Optimal total volume~\cite{cahill2023optimal} of Parseval frames is also a related concept.) Equal-norm tight frames are optimal under certain models to noise and single erasures~\cite{GKK01} and are in turn $1$-uniform. ETFs are optimal to two erasures and are $2$-uniform~\cite{BP05,StH03,HoPa04}. ETFs with all zero, all negative, or all purely imaginary triple products are optimal to $3$ erasures and $3$-uniform~\cite{BP05,hoffman2012complex}. Symmetry and  $k$-uniformity are related, with $k$-fold symmetry almost implying $k$-uniformity, as outlined below.
It is known that a \emph{Fourier frame}, i.e., a frame formed as an all-column submatrix of a Fourier transform matrix is always an equal-norm tight frame with transitive vector symmetry group containing the group of modulations.  More generally, the orbit of a non-zero vector under a finite irreducible unitary representation is always an equal norm tight frame~\cite{vale2005tight} with transitive vector symmetry group. There are other results concerning when orbits of a non-zero vector under a projective unitary representation yields an equal-norm tight frame with necessarily transitive line symmetry group (see, e.g.,~\cite{lawrence2005linear,av2025abelian}). It is, however, possible for a spanning set to have a transitive vector or line symmetry group and not be a tight frame.  For example, if one multiplies the first entry of every vector in a Paley ETF $\Phi_p$ by a non-unimodular scalar, the resulting set of vectors form a spanning set with a transitive vector/line symmetry group that contains cyclic permutations induced from the Paley modulation, but the new vectors would no longer be a tight frame / $1$-uniform frame.
Any spanning set which has $2$-transitive line symmetry group is an ETF~\cite{iverson2024doubly}, and it immediately follows from the definition of $2$-homogeneity that any equal-norm tight frame which has $2$-homogeneous line symmetry group is an ETF, where ETFs are precisely the $2$-uniform frames.  The proof of Lemma~\ref{lem:homcomp} shows that any $3$-homogeneous ETF must be $3$-uniform. ETFs which have purely imaginary triple products must have $n=2d$~\cite{ettaoui2000equiangular,strohmer2008note} and always exist when $2d \times 2d$ skew-symmetric conference matrices exist~\cite{delsarte1991bounds}.  However, we will see in Theorem~\ref{thm:khomETF} that only certain skew-symmetric conference matrices will lead to $3$-homogeneous ETFs.

As we will see in Theorem~\ref{thm:khomETF}, $k$-homogeneity of ETFs for $k>1$ only holds in a few interesting cases. For $k \geq 4$, $k$-homogeneity basically only holds for orthonormal bases and simplices or because $(n-k)$-homogeneity implies $k$-homogeneity.  The remaining cases to study are $2$-homogeneous ETFs which are not $2$-transitive (otherwise one could look to \cite{IvMi18,dempwolff2023trans,iverson2024doubly}) and $2$-transitive ETFs which are $3$-homogeneous.  The latter case is completely characterized below.

We note that Gerzon's bound~\cite{LS1973} yields that if $n > d^2>1$, there does not exist a $(d,n)$-ETF.
Combining this fact with Naimark complementation means that $\max\{ d^2, (n-d)^2\} \geq n$ is a necessary condition for the existence of a nontrivial $(d,n)$-ETF. 
Thus, the only parameters $(d,n)$ with $n < 8$ which could potentially yield a non-trivial ETF are
\[
(2,4), \, (3,6),\, (3,7),\, (4,7);
\]
in fact, ETFs exist~\cite{FM15}  for all such parameters, sometimes uniquely.
The $(2,4)$-SIC is the unique (up to switching equivalence and vector permutation, cf.~\cite{zhu2010sic}) ETF of $4$ vectors in $\bC^2$ and the real $(3,6)$-ETF is the unique (up to switching equivalence and vector permutation, cf.~\cite{BKo06}) real $(3,6)$-ETF.

\begin{thm}\label{thm:khomETF}
The only nontrivial ETFs with $3$-homogeneous line symmetry groups are the $(2,4)$-SIC and the $(d,2d)$-ETFs with  $d = q+1$ where $q\equiv 3 \bmod 4$ is a prime power arising from Paley skew-conference matrices, which are all $2$-transitive. The only nontrivial ETFs with $4$-homogeneous line symmetry groups are the $(2,4)$-SIC and the real $(3,6)$-ETF.
\end{thm}
\begin{proof}
We consider $(d,n)$-ETFs. By \cite{livingstone1965transitivity}, if $n \geq 2k$, then any $k$-homogeneous group action on a set of size $n$ is $(k-1)$-transitive. 
Thus, all $3$-homogeneous ETFs (lines or vectors) with $n\geq 6$ are also $2$-transitive.  
By Lemma~\ref{lem:3hom}, the only nontrivial ETFs which have $3$-homogeneous line symmetry group must have $n=2d$.  
It follows from~\cite{iverson2024doubly} that the only $2$-transitive ETFs with $n=2d \geq 6$ and purely imaginary inner products have $d = q+1$ with $q\equiv 3 \bmod 4$ and are generated by Paley skew-conference matrices. 
These have line symmetry group $\operatorname{PSL}(2,q)$ acting necessarily faithfully and transitively, where (see, e.g.,~\cite{kantor1972khomo}) $\operatorname{PSL}(2,q)$ acts $3$-homogeneously. 
The only nontrivial ETF with $n<6$ is the $(2,4)$-SIC, which is $1$-homogeneous as a SIC and thus $3=(4-1)$-homogeneous as well.  
Note that the $(2,4)$-SIC has $2$-transitive line symmetry group \cite{Zhu15}; thus, all of the ETFs with $3$-homogeneous line symmetry group are also $2$-transitive.

Applying \cite{livingstone1965transitivity} again, we see that all $4$-homogeneous ETFs (lines or vectors) with $n\geq 8$ are also $3$-transitive, which means they are trivial (Lemma~\ref{lem:trip}).  Similarly, trivial ETFs with $4\leq n<8$ trivially have $4$-homogeneous line symmetry groups (Lemma~\ref{lem:sn}). The remaining ETFs which have $n < 8$ are the $(2,4)$-SIC, the $(3,6)$-ETFs, the $(3,7)$-ETF(s) and $(4,7)$-ETF(s). The $(2,4)$ SIC is trivially $4$-homogeneous since the identity maps the group of four lines to the group of four. If a (not necessarily real) $(3,6)$-ETF is $4$-homogeneous, then it is $2$-homogeneous. Since $6$ is not prime power, any $2$-homogeneous $(3,6)$-ETF must be $2$-transitive~\cite{kantor1969auto}. The real $(3,6)$-ETF corresponds to the antipodal vertices of the regular icosahedron and thus has line symmetry group isomorphic to $\operatorname{PSL}(2,5)$, which acts $2$ transitively on the lines (see, e.g., Theorem 1.3 \cite{iverson2024doubly} which also gives uniqueness up to equivalence of the $(3,6)$-ETFs with $2$-transitive symmetry), in turn implying that the automorphism group is $2$-homogeneous and thus $4=(6-2)$-homogeneous. The remaining ETF parameters are the $(3,7)$- and $(4,7)$-ETFs. We have already shown that neither is $3$-homogeneous; thus, neither can have a $4=(7-3)$-homogeneous line symmetry group.
\end{proof}

We note that in the theory of $k$-homogeneous groups, $k > n/2$ is usually excluded in hypotheses as $(n-k)$-homogeneity is in a sense calling the shots.
So, one might rephrase the second half of the statement of Theorem~\ref{thm:khomETF} as, ``There are no nontrivial ETFs which are nontrivially $4$-homogeneous.''
As with the characterization of $4$-homogeneity, in the remaining cases where $k > 4$, any $k$-homogeneous ETF with $n \geq 2k$ will be $(k-1)$-transitive and thus can only be a simplex ETF or orthonormal basis.  For $n < 2k$, the ETF will be $(n-k)$-homogeneous with $n-k < k$.  Thus, inductively, the only possibilities are $1$-homogeneous ETFs (i.e., the ETF lines are spun up as the orbit of a group with no additional symmetries), $2$-homogeneous ETFs, or any of the $3$ or $4$-homogeneous ETFs enumerated above.

There are skew-conference matrices which are not Paley (e.g.,~\cite{Handbook}).  Thus, being an ETF with purely imaginary triple products is not sufficient to be $3$-homogeneous.

\begin{cor}
The only $(d,n)$-ETFs with $3$-homogeneous vector symmetry groups are orthonormal bases, the regular simplex ETFs, and the Naimark complements of regular simplex ETFs with $n\geq 3$. 
The only $(d,n)$-ETFs with $4$-homogeneous vector symmetry groups are the $(2,4)$-SIC and the $3$-homogeneous ETFs with $n \geq 4$.
\end{cor}
\begin{proof}
By following the cases in the proof of Theorem~\ref{thm:khomETF}, we see that the only possible ETF which has a $3$-homogeneous vector symmetry group which is not $2$-transitive is the $(2,4)$-SIC.  Although any SIC is the orbit of a fiducial under a unitary system, the unitary system is a projective unitary representation and thus does not act transitively on the vectors.  Hence, the vector symmetry group of the $(2,4)$-SIC is neither $1$-homogeneous nor $3=(4-1)$-homogeneous (Lemma~\ref{lem:homcomp}). All other ETFs with ETFs with $3$-homogeneous line symmetry groups must be $2$-transitive, and hence Lemma~\ref{lem:trip} applies.

By Theorem~\ref{thm:khomETF}, we just need to check the $(2,4)$-SIC and the $(3,6)$ real ETF for $4$-homogeneous automorphism vector groups.  The $(2,4)$-SIC is trivially $4$-homogeneous. If the $(3,6)$ real ETF had a $4$-homogeneous vector symmetry group, then (Lemma~\ref{lem:homcomp}) it would be $2=(6-4)$-homogeneous, which (Lemma~\ref{lem:trip}) cannot happen since it is neither an orthonormal basis or a regular simple ETF.
\end{proof}

\section{Dave Larson's AMS Memoirs}\label{sec:memoirs}
The preceding results mirror several themes in~\cite{HL00,DaL98}; namely, that group actions are a powerful tool to construct frames with desirable properties, that Naimark complementation can be used to simplify the search for frames with particular properties, and that associated groups other than those used to generate the frame are powerful tools in understanding the frame properties.

Many infinite frames of note---like Gabor/time-frequency and wavelet---are generated as the (subsampling) of the orbit of a (projective) unitary representation and were a main focus of Dave Larson's two AMS Memoirs~\cite{HL00,DaL98}.  A Gabor frame with time-frequency shifts over a lattice would be an example of an infinite frame with a transitive automorphism group of lines.   We consider for a moment automorphism groups of such frames over infinite dimensional (separable) Hilbert spaces.  Rephrasing Theorem 3.8$'$ in~\cite{HL00}: If a tight frame $\Phi$ is the orbit of a unitary representation of a group $H$, then the Naimark complement is also the orbit of some unitary representation of the same group. Such a group $H$ is a subgroup of the automorphism groups of the lines and the vectors; so, this result is related to the fact that Theorem~\ref{thm:TP} implies that the automorphism groups of a finite tight frame and a Naimark complement are the same. Also, Corollary 3.10 of~\cite{HL00} yields that if a tight frame with frame bound $1$ is the orbit of a single bounded unitary operator, then it is unitarily equivalent to $\{e^{ 2\pi i n \cdot}\vert_E \, : \, n \in \bZ\}$ for $E \subset \bT$, which is related to submatrices of Fourier matrices always yielding tight frames.

Assume that  $\Phi = (\varphi_j)_{j\in J}$ is a tight frame with frame bound $A$ for an infinite dimensional separable Hilbert space $\cH$. Then $J$ is countably infinite. If $\Phi$ were $2$-homogeneous, then $\absip{\varphi_i}{\varphi_j} = \absip{\varphi_k}{\varphi_{\ell}}$ for all $i \neq j$ and $k \neq \ell$.  Call this number $a$. By the definition of a tight frame (Definition~\ref{defn:frame}), we see that
\[
\norm{\varphi_k}^2 = \frac{1}{A} \sum_{j \in J} \absip{\varphi_k}{\varphi_j}^2 =\frac{1}{A}\left( \norm{\varphi_k}^2 + \sum_{j \neq k}a^2\right).
\]
Since the left-hand-side is a finite number, the right-hand-side must converge, meaning that $a = 0$.  Thus, the only infinite $2$-homogeneous infinite frames are orthogonal bases.

An important concept introduced in~\cite{DaL98} (with additional contemporary publications \cite{DaL98,hernandez1996smoothing,hernandez1997smoothing,fan1996construction,dai1997wavelet,dai1997wavelet2}) is wavelet sets.  \emph{(Parseval) wavelet sets} are measurable subsets of $\bR^d$ for which the integer (lattice) translates satisfy one condition and the dyadic dilations satisfy another condition.  When $X$ is a (Parseval) wavelet set, then the inverse Fourier transform of the indicator function of $X$ generates a (wavelet tight frame with frame bound $1$) a wavelet orthonormal basis.  The simple geometric characterization of wavelet sets allowed the discovery of wavelets with surprising features.  We may view the construction of Fourier ETFs as a kindred concept.  By instead zeroing out rows rather than deleting them, we could restate Theorem~\ref{thm:diffsetcon} as meaning that modulating an indicator vector of a set yields a ETF for the span if and only if the set is a difference set in the group used to generate the Fourier matrix. As in Example~\ref{ex:DSBIBD}, taking all of the shifts of any difference set (called the \emph{development}) always yields a BIBD.

Finally, we note that one of the key tools developed in \cite{DaL98} was \emph{local} (to a frame vector) \emph{commutants}.  This led to the analysis of the von Neumann algebra generated by the orbit of the left regular representation of a group $G$ being a tool to understand whether other representations of $G$ can yield frames \cite{HL00}.  A related idea is the use of Schurian association schemes to understand finite ETFs in \cite{IvMi18,iverson2024doubly}.  It is possible that the similarities between these approaches could be leveraged to create further tools to characterize ETFs.

\subsection{Problems}
In honor of the ``Problems'' offered in Dave Larson's Memoirs \cite{HL00,DaL98}, we consider a handful of questions that this manuscript leads to.

Corollary~\ref{cor:BIBD} gives that the bender of any Paley ETF is a BIBD.  However, it does not give information on the parameters of the BIBD. The bender of a Paley ETF $\Phi_p$ for $p$ a prime is simply the set of all subsets of size $(p+1)/2$. Calculations from Harley Meade show that the spark of $\Phi_{27}$ is $8$ and the bender is a $(27,8,28)$-BIBD. Calculations for Paley ETFs with larger $q$ were not tractable. 
\begin{question}\label{quest:bender}
What are the spark and bender of the Paley ETFs $\Phi_q$ with non-cyclic additive groups?
\end{question}
Dangerously powered by one infinite class and one example outside of this class, we propose the following conjecture.
\begin{conj}
Let $\Phi_q$ be a Paley ETF with $q=p^s$.  Then the spark of $\Phi_q$ is $((p+1)/2)^s$ and the bender is a $(p^s, ((p+1)/2)^s,\lambda)$-BIBD.
\end{conj}
To aid in solving the conjecture, there is code~\cite{EmilyPaley} to allow one to generate Paley ETFs of arbitrary size.  This code also generates the symmetries $T$ and $\Pi$.  There is additionally a function to construct ETFs from any difference sets.

We note that the computation of the Gram matrix of $\Phi_p$ for $p$ prime in the proof of Lemma~\ref{lem:paleyTP} shows that the $\Phi_p$ are neither of two other categories of ETFs which are in some sense $2$-transitive-adjacent, i.e., being DRACKN or roux \cite{CGSZ16,IvMi18}. This leads to the following question.
\begin{question}
Are there any ETFs which are simultaneously $2$-homogeneous and DRACKN or $2$-homogeneous and roux but not $2$-transitive?
 \end{question}

 It is a fact that any linear isometry of $\bC^d$ with respect to $\ell^p$ with $p\neq 2$ is a \emph{generalized permutation matrix}, i.e., a product of a diagonal unitary and a permutation matrix \cite{li1994isometries}.  It is also a fact that any element $U$ of the automorphism group of a $(d,n)$-Fourier frame necessarily satisfies $\norm{U\varphi_j}_p = \norm{\varphi_j}_p$ for all $1 \leq p \leq \infty$ with the $\varphi_j$ being a spanning set for $\bC^d$. Being an isometry for $p \neq 2$ on a spanning set is not sufficient to prove that it extends to an isometry on the entire space.  However, the computed automorphisms in this paper are all generalized automorphisms.  
 \begin{question}
 Are the elements of the automorphism group of a Fourier ETF always generalized permutations?
 \end{question} 
We point to code from Joey Iverson~\cite{JoeySymm} to determine the vector and line symmetry groups of frames as a resource to help answer this question.

\section*{Acknowledgments}
The author is especially grateful to Joey Iverson, for fruitful discussions and GAP code \cite{JoeySymm} to run some experiments.  
The author is also thankful to Harley Meade, who ran Joey's code to help formulate hypotheses and explicitly proved that groups related to $\{M^k T^{\ell}\}$ act $2$-homogeneously on ETFs related to prime-order Paley ETFs, to Dustin Mixon for helpful feedback, and to the anonymous referees for their thoughtful suggestions which lead to great improvements in the text. 
Notably, one referee suggested an approach that greatly streamlined the proof of Theorem~\ref{thm:Paleydbhm}; however, this new approach obfuscates the connection of the material to the Matlab implementation in~\cite{EmilyPaley}.
See Section 2 of Version 2 of this paper on arXiv to understand the connection.
Most of all, the author is grateful to Dave Larson for introducing her to frames twenty-something years ago.


\newcommand{\etalchar}[1]{$^{#1}$}
\providecommand{\bysame}{\leavevmode\hbox to3em{\hrulefill}\thinspace}
\providecommand{\MR}{\relax\ifhmode\unskip\space\fi MR }
\providecommand{\MRhref}[2]{%
  \href{http://www.ams.org/mathscinet-getitem?mr=#1}{#2}
}
\providecommand{\href}[2]{#2}

\end{document}